\documentclass[12pt,reqno]{amsart}
\usepackage{amsfonts,amssymb,amsmath}
\usepackage{xcolor}
\newcommand\blue[1]{\textcolor{blue}{#1}}

\newcommand\magenta[1]{\textcolor{magenta}{#1}}

\newcommand\violet[1]{\textcolor{violet}{#1}}
\definecolor{fgreen}{RGB}{44,144, 14}

\renewenvironment{proof}{{\bfseries Proof.}}{\qed}

\numberwithin{equation}{section} 
\newtheorem{theorem}{Theorem}[section] 
\newtheorem{proposition}[theorem]{Proposition} 
 
\newtheorem{lemma}[theorem]{Lemma} 

\theoremstyle{definition}
\newtheorem{definition}[theorem]{Definition} 
 
\newtheorem{example}[theorem]{Example}

\usepackage[pdftex,colorlinks, hypertexnames=false,  
linkcolor=blue,urlcolor=blue,citecolor=cyan]{hyperref}

\def\R{\mathbb R}
\def\s{\mathbb S}
\def\C{\mathbb C}

\def\F{\mathbb F}
\def\H{\mathbb H}

\def\ib{\mathbf {i}}
\def\jb{\mathbf {j}}
\def\kb{\mathbf {k}}

\def\d{\mathbf{ d}}

\def\N{\mathbb N}

\def\E{\mathbb E}

\def\R{\mathbb R}

\def\d{\mathcal D}
\def\g{\mathcal G}
\newcommand{\SL}{\mathrm{SL}}

\newcommand{\GL}{\mathrm{GL}}

\def\M{{\mathrm M}}
\def\d{\mathcal D}

\def \x {{\bf x}}

\def\R{\mathbb {R}}
\def\C{\mathbb {C}}
\def\N{\mathbb {N}}
\def\H{\mathbb {H}}

\def\E{\mathbb {E}}
\def\F{\mathbb {F}}

\def\d{\mathbf {d}}

\def\ib{\mathbf {i}}

\def\jb{\mathbf {j}}

\def\ZC{\mathcal {Z}}

\def\PC{\mathcal {P}}

\def\BC{\mathcal {B}}

\def\g{\mathfrak {g}}

\def\s{\mathfrak {s}}

\def\l{\mathfrak {l}}

\def\lto{\longrightarrow}

\newcommand{\secref}[1]{Section~\ref{#1}}
\newcommand{\thmref}[1]{Theorem~\ref{#1}}
\newcommand{\lemref}[1]{Lemma~\ref{#1}}

\begin{document}

\title[Real adjoint orbits of special linear groups]{Real adjoint orbits of special linear groups}
\author[K. Gongopadhyay,   T.  Lohan and C Maity ]{Krishnendu Gongopadhyay, Tejbir Lohan and Chandan Maity}

\address{Indian Institute of Science Education and Research (IISER) Mohali,
	Knowledge City,  Sector 81, S.A.S. Nagar 140306, Punjab, India}
\email{krishnendu@iisermohali.ac.in, krishnendug@gmail.com}

\address{Indian Institute of Technology Kanpur, Kanpur 208016, Uttar Pradesh, India}
\email{tejbirlohan70@gmail.com}

\address{Indian Institute of Science Education and Research (IISER)  Berhampur, Berhampur
760010, Odisha, India }
\email{cmaity@iiserbpr.ac.in }

\makeatletter
\@namedef{subjclassname@2020}{\textup{2020} Mathematics Subject Classification}
\makeatother



\begin{abstract}
Let $ G $ be a linear Lie group with Lie algebra $ \g $. An element $ X \in \g$ is called $\mathrm{Ad}_G$-real if  $ -X=gXg^{-1} $ for some $ g\in G $. Moreover, if $ -X=gXg^{-1} $  holds  for some involution $ g\in G $, then $ X $ is  called strongly ${\rm Ad}_G$-real.  We have classified the $\mathrm{Ad}_G$-real and strongly $\mathrm{Ad}_G$-real orbits in the special linear Lie algebra $ \s\l(n,\F) $ for $ \F=\C$ or $\H $.
\end{abstract}

\maketitle

\section{Introduction}
A group element that is conjugate to its inverse is known as a \emph{real} element in the group. Such an element is also known as a \emph{reversible} element. The notion is closely related to the `strongly reversible' or `strongly real' elements in a group, which are product of two involutions (i.e., elements of order at most two) in the group. The real elements or the real conjugacy classes in a group have seen many investigations in the literature from different perspectives. The terminology `reversible' has been motivated by dynamical (or geometric) considerations, whereas the terminology `real' has been motivated by a classical theorem of Frobenius and Schur that states that the number of real irreducible characters of a finite group $ G $ is equal to the
number of reversible conjugacy classes of $ G $.
The latter viewpoint is of importance in representation theory. Despite many works, complete classifications of real and strongly real elements are known only for a very few families of Lie groups. Mostly, the equivalence between real and strongly real classes has been understood for certain groups, e.g.,  \cite{W, ST}.
A complete classification of such elements is known only for a very few families. We refer to the monograph \cite{FS} for a survey.

Here we ask a related problem that comes from the natural action of a Lie group $G$ on its Lie algebra $\g$ by adjoint representation. When $G$ is a linear Lie group, the above action is given by conjugation, i.e., $ {\rm Ad}(g)X:=gXg^{-1}$.   An element $X\in \g$ is called $ {\rm Ad}_G$-real  if $-X =gXg^{-1} $ for some $g\in G$.  An  $ {\rm Ad}_G$-real element $X$ is called strongly $ {\rm Ad}_G$-real if $-X = \tau X \tau^{-1} $ for some involution   $\tau\in G$, compare \cite[Definition 1.1]{GM}. Observe that if $X\in \g$ is  $ {\rm Ad}_G$-real, then $\exp X$ is real in $G$. This notion may be thought of as an infinitesimal analogue of the classical reality for Lie algebras. The investigation of the adjoint orbits in a semi-simple Lie group is an active area of research in differential geometry; see the survey \cite{CoM} and relatively recent article \cite{Mc}.
The above problem may be restated as the problem of classifying the `real' and `strongly real' adjoint orbits of semi-simple Lie groups.
There are many directions of research on adjoint orbits; see \cite{ABB} for a symplectic geometry point of view, \cite{CM, BCM, BCM2} for a topological point of view. Nevertheless, we have not seen any literature where such reality or reversibility problems have been addressed so far.

For a Lie group $G$ with Lie algebra $\g$,  a natural problem following \cite[Definition 1.1]{GM} is classification of the $ {\rm Ad}_{{ G }}$-real elements in $\g$. The unipotent real and strongly real classes in classical Lie groups have been classified in \cite{GM} via classification of the  $ {\rm Ad}_{{ G }}$-real and strongly $ {\rm Ad}_{{ G }}$-real nilpotent elements in classical Lie algebras. This notion has also been used in \cite{GLM1} to classify the reversible and strongly reversible affine transformations.

Let $\F$ denote the complex numbers $\C$ or the real quaternions $\H$. It is a fact that if $ X\in \g\l(n,\F) $ is  $ {\rm Ad}_{{\rm GL }(n,\F)}$-real, then $ X $ is an element of $ \s\l(n,\F) $.   Also for the elements in $ \s\l(n,\F) $, the $ {\rm Ad}_{{\rm SL }(n,\F)} $-reality and the $ {\rm Ad}_{{\rm GL }(n,\F)} $-reality are equivalent; see \lemref{equi-SL-GL}. 
These motivate us to investigate the Ad-reality problem for the adjoint action of the special linear groups in this paper. We  have  classified  the  $ {\rm Ad}_{{\rm SL}(n,\F)} $-real and the strongly  $ {\rm Ad}_{{\rm SL}(n,\F)} $-real elements in  the Lie algebra $\frak{sl}(n, \F)$. 
We hope the ideas used here will motivate similar classification in other Lie algebras.

Sometimes, we shall drop the `$\mathrm{Ad}$' and henceforth call the above type of elements as real elements,  resp.  strongly real elements in the respective Lie algebras. To state our main results, we need some terminology concerning the partition of $n$. 
A \textit{partition} of a positive integer $n$ is an object of the form $ {\d}(n) := [ d_1^{t_{d_1}}, \dots,   d_s^{t_{d_s}} ],$ where $t_{d_i}, d_i \in \N, 1\leq i\leq s, $ such that $ \sum_{i=1}^{s} t_{d_i} d_i = n, t_{d_i} \geq 1 $ and  $ d_1 >  \cdots > d_s > 0$. 
For a partition $\d (n)\, =\, [ d_1^{t_{d_1}},\, \ldots ,\, d_s^{t_{d_s}} ]$ of $n$, define 
\begin{align}
	{\N}_{\d(n)}  :=\, \{ d_i \,\mid\, 1 \,\leq\, i \,\leq\, s \}\, \quad &{\rm and } \quad
	{\E}_{\d(n)}:=\,  {\N}_{\d(n)}   \cap  2\N  
	\label{definition-Nd-Ed}\\ 
	\E_{{\d}(n)}^2  := \{ \eta \in \E_{{\d}(n)} \,&\mid\, \eta \equiv 2   \pmod  4   \}. \nonumber
\end{align} 
Let $ \mathcal{P}(n) $ be the \textit{set of all partitions of $n$}. A partition  $ {\d}(n) $ of  $n$ will be called \textit{even} if $d_i$ is even for all $1 \leq i \leq s$, i.e., $ \E_{{\d}(n)}\,=\, \N_{{\d}(n)} $. 
Let $ \mathcal{P}_{\rm even}(n) $  be the subset of $ \mathcal{P}(n) $ {\it consisting of all even partitions of $n$}.  We call a partition  $ {\d}(n) $ of $n$ to be   \textit{very even} if ${\d}(n)$ is even, and $t_\eta$ is even for all $ \eta \in \N_{{\d}(n)}. $  Let $ \mathcal{P}_{\rm v.even}(n) $ be the subset of $ \mathcal{P}(n) $ consisting of all very even partitions of $n$. 
Now define  
\begin{align}\label{definition-P-e-tilde}
	\widetilde{ \mathcal{P}_e}(n) := \{  {\d}(n) \in  \mathcal{P}_{\rm even}(n)  \backslash  \mathcal{P}_{\rm v.even}(n)  \mid \sum_{\eta \in \E_{{\d}(n)}^2}  t_\eta \ \hbox {is  odd} \ \}.
\end{align}   
Equivalently,  $ \widetilde{ \mathcal{P}_e}(n)$ denote the set of all partitions ${\d}(n)$ of $n$ such that:
\begin{itemize}
	\item $ \N_{{\d}(n)} \,=\, \E_{{\d}(n)}$,
	\item $ {\d}(n)$ is such that not all $t_{d_i}$ corresponding to $d_i$ are even,
	\item $  \sum_{d_i \in \E_{{\d}(n)}^2}  t_{d_i}$ is odd.
\end{itemize}

With the above notation, our first result classifies  ${\rm Ad}_{{\rm SL }(n,\F)} $-real elements in $\frak{sl}(n, \F)$. 
This theorem is a combination of \thmref{thm-real-sl(n,C)} and \thmref{thm-real-sl(n,H)}.
\begin{theorem}\label{thm-real}
	Let $X \in \mathfrak{sl}(n,\F)$ be any arbitrary element. Then $X $ is ${\rm Ad}_{{\rm SL }(n,\F)} $-real if and only if both the following conditions hold:
	\begin{enumerate}
		\item \label{cond-1-real-sl(n,H)-thr-intro} 
		Let $\lambda$ represent an eigenvalue class of $X$ such that 
		
		\medskip {\rm (a)}  the real part of $\lambda$ is non-zero  for the case  when $ \F=\H $, 
		
		\medskip	{\rm (b)}  $\lambda \neq 0$ for the case  when $ \F=\C $.
		
		\medskip 
		Then, $-\lambda$ also represents an eigenvalue class of $X$ with the same multiplicity. 
		\item \label{cond-2-real-sl(n,H)-thr-intro} Let $ m $ be the multiplicity of  the eigenvalue classes represented by $ \lambda $ and $ -\lambda $. 
		Let $ \d_\lambda $  and $ \d_{-\lambda} $ be the partitions of $ m $ associated to the eigenvalue classes $ \lambda $ and $ -\lambda $, respectively. Then 
		$$ \d_\lambda\,=\, \d_{-\lambda}  \,.$$
	\end{enumerate} 
\end{theorem}

The classification of strongly ${\rm Ad}_{{\rm SL }(n,\F)} $-real elements in $ \mathfrak{sl}(n,\F)$ is more subtle, and the results are different depending upon the arithmetical properties of the underlying skew-fields. The following result gives a classification of the strongly ${\rm Ad}_{{\rm SL }(n,\C)} $-real elements in $ \mathfrak{sl}(n,\C)$.

\begin{theorem}\label{thm-strongly-real-sl(n,C)}
	Let $X \in \mathfrak{sl}(n,\C)$ be a non-zero arbitrary $ {\rm Ad}_{{\rm SL}(n,\C)} $-real element.  Let $0$  be an eigenvalue of $X$ with the multiplicity  $p_o$,   $  p_o \geq 0$,
	and $ {\d}(p_o) \in \mathcal{P}(p_o) $ be the corresponding partition associated to the eigenvalue $ 0 $. Then $X$ is strongly real if and only if  either $0$ is an eigenvalue of $X$ such that $ {\d}(p_o) \not\in  \widetilde{ \mathcal{P}_e}(p_o)$ or $n \not\equiv 2 \  (mod \ 4 \ )$.
\end{theorem}

For $\mathfrak{sl}(n,\H)$, we have the following classification of strongly ${\rm Ad}_{{\rm SL }(n,\H)} $-real elements.  
\begin{theorem}\label{thm-strong-real-sl(n,H)}
	Let $X \in \mathfrak{sl}(n,\H) $ be an arbitrary non-zero ${\rm Ad}_{{\rm SL }(n,\H)} $-real element. Then the following statements are equivalent.
	\begin{enumerate}
		\item  $X$ is strongly ${\rm Ad}_{{\rm SL }(n,\H)} $-real.
		\item \label{cond-2-main-thrm-sl(n,H)-str} Let $\mu$ be the non-zero purely imaginary complex representative of an eigenvalue class of $X$ with the multiplicity $m_{\mu}$ and $\mathbf{d}({m_\mu}) = [ d_1^{t_{d_1}}, \dots,   d_s^{t_{d_s}} ] \in \mathcal{P}(m_{\mu}) $ be the corresponding partition. Then $t_{d_\ell}$ is even for all $ 1 \leq \ell \leq s$.
	\end{enumerate}
\end{theorem}

In the proofs of the above results, we have explicitly constructed conjugating elements for the Ad$ _{{\SL}(n,\F)}$-real and the strongly  Ad$_{{\SL}(n,\F)}$-real elements in $ \s\l(n,\F) $. Using the centralizer of such an element, one can get all the conjugating elements. In some cases,  characterizations of the conjugating elements are given.

The above results are proven by analyzing the linear action of the Lie groups on their respective Lie algebras. The Jordan canonical form is one of the key tools that help us to identify canonical elements in each conjugacy class. For this reason, we have restricted ourselves to $\mathbb{C}$ and $\mathbb{H}$. In \cite[Section 3.5]{GM}, a suitable ordered basis of the underlying vector space corresponding to a given nilpotent element has been described. Such an ordered basis is crucially used in classifying the strongly ${\rm Ad}_G$-real elements.

The paper is organized as follows. We fix notation and recall some preliminary results in \secref{sec-Notation}. Some useful lemmas have been proved in \secref{sec-lemma}. These lemmas are used to prove the main theorems. Finally, \thmref{thm-real-sl(n,C)}, \thmref{thm-real-sl(n,H)}, \thmref{thm-strongly-real-sl(n,C)}, and \thmref{thm-strong-real-sl(n,H)} are proven in \secref{sec-pf-thm}.

\section{Notation and Background}\label{sec-Notation} 
Here, we fix some notation and summarize some basic concepts and known results.
The Lie groups will be denoted by the capital 
letters, while the Lie algebra of a Lie group will be denoted by the corresponding lowercase fracture letter.
Let $ G $ be a linear Lie group with Lie algebra $ \g$, and $ X\in \g$. The {\it centralizer} of $X$ in $G$ is 
$$\ZC_{G} (X)\,:=\, \{h\,\in\, G \,\,\big\vert\,\,
hXh^{-1} \,=\, X\,  \}.$$

The notation $\F$  stands for either $\C$ or $\H$, where $\C$ denotes the field of complex numbers, and $\H$ denotes the division ring of Hamilton's quaternions. We recall that every element in $\H$ can be expressed as $a=a_0 + a_1 \ib + a_2 \jb + a_3 \kb$, where $a_0$, $a_1$, $a_2$, $a_3$ are real numbers,  and 
$\ib^2=\jb^2=\kb^2=-1$,  $\ib\jb=-\jb\ib=\kb$,  $ \jb\kb=-\kb\jb=\ib$, $\kb\ib=-\ib\kb=\jb$. The conjugate of $a$ is given by $\bar a=a_0-a_1\ib -a_2\jb -a_3\kb$.   We identify the real subspace $\R \oplus \R \ib$ with the usual complex plane $\C$, and then one can write $\H= \C \oplus  \C \jb$.  

The following notation will allow us to write block-diagonal square matrices with many blocks in a convenient way.
For $r$-many square matrices $A_i \,\in\, {\rm M}({m_i}, \F)$, $1 \,\leq\, i \,\leq\, r$,  the block diagonal square
matrix of size $\sum m_i \times \sum m_i$, with $A_i$ as the $i$-th 
block in the diagonal, is denoted by $A_1 \oplus \cdots \oplus A_r$. This is also abbreviated as  $\oplus_{i =1}^r A_i$.  For two disjoint ordered sets $(v_1,\, \ldots ,\,v_n)$ and $(w_1,\, \ldots ,\,w_m)$,  the ordered set $(v_1,\, \ldots ,\,v_n,\, w_1,\, \ldots ,\, w_m)$ will be denoted by
$$
(v_1, \,\ldots ,\,v_n) \vee (w_1,\, \ldots ,\,w_m)\, .
$$

\subsection{Matrices over the quaternions}
For an elaborate discussion on the theory of quaternions and matrices over the quaternions;  see \cite{rodman, FZ}.

\begin{definition}\label{def-eigen-M(n,H)}
	Let $A \in  \mathrm{M}(n,\H)$,  the algebra of $n \times n$ matrices over $\H$.  A non-zero vector $v \in \H^n $ is said to be a (right)  eigenvector of $A$ corresponding to the (right) eigenvalue  $\lambda \in \H $ if the equality $ Av = v\lambda $ holds.
\end{definition}

Eigenvalues of $A$ occur in similarity classes. If $v$ is an eigenvector corresponding to an eigenvalue $\lambda$,  then $v \mu \in v \mathbb H$ is an eigenvector corresponding to the eigenvalue $\mu^{-1} \lambda \mu$. Each similarity class of eigenvalues contains a unique pair of complex numbers which are complex conjugate to each other. Often we shall refer to them as `eigenvalues', though it should be understood that our reference is towards their similarity classes. In places where we need to distinguish between the similarity class and a representative, we shall write the similarity class of an eigenvalue representative $\lambda$ by $[\lambda]$. We shall mostly choose the \textit{unique complex representative} of a similarity class that has non-negative imaginary part.

Now, we recall the relationship between quaternion and complex matrices. Let $A \in  \mathrm{M}(n,\H)$ and write $ A = (A_1) + (A_2) \jb $,  where $ A_1,  A_2 \in \mathrm{M}(n,\C)$. Consider the embedding 
\begin{align}\label{definition-Phi-embedding}
	\Phi\colon  \mathrm{M}(n,\H)\, \lto \,  \mathrm{M}(2n,\C)\,; \quad \Phi(A):= \begin{pmatrix}
		A_1   &  A_2 \\
		- \overline{A_2} & \overline{A_1}  \\ 
	\end{pmatrix}.
\end{align}
The {\it determinant} and {\it trace}  of $A$ is defined to be
$ {\rm det}_{\H}(A) :=  {\rm det}(\Phi(A))$ and ${\rm tr }_\H( A) := {\rm tr }(\Phi(A))$, respectively. The above definitions of determinant and trace of a matrix over $\H$ are well-defined due to the Skolem-Noether theorem. For $ A\in {\M}(n, \F) $, let $ A^T $ denote the {\it transpose } of the matrix $ A $, where $\F= \C $ or $\H$.  Let $ \mathrm{I}_n$ denote the identity matrix of order $ { n \times n} $,   and $\mathrm{ O}_{ m \times n} $ denote the zero matrix of order $ { m \times n} $.
Consider the Lie groups 
$$ \mathrm{SL}(n,\C) = \{ g \in \mathrm{GL}(n,\C) \mid \det(g) = 1 \} \,,\quad  \mathrm{SL}(n,\H) = \{ g \in \mathrm{GL}(n,\H) \mid {\det}_{\H}(g)  = 1 \} $$
with the associated  Lie algebras 
$$\mathfrak{sl}(n,\C)  = \{ X \in  \mathrm{M}(n,\C) \mid {\rm tr }(X) = 0 \}\,, \quad \mathfrak{sl}(n,\H)  = \{ X \in  \mathrm{M}(n,\H) \mid {\rm tr }_\H( X)  = 0 \}\,.$$
The natural adjoint representation   $ {\rm Ad} : \mathrm{SL}(n,\F) \lto \mathrm{GL}( \mathfrak{sl}(n,\F) )$ is given by ${\rm Ad}(g)X := gXg^{-1}$ for all $ X \in \mathfrak{sl}(n,\F) $, $ g \in \mathrm{SL}(n,\F)$.

\begin{definition}[cf.~{\cite[p.  94]{rodman}}] \label{def-jordan}
	A \textit{Jordan block} $ \mathrm{J}(m,\lambda)$ is a $m \times m$ matrix with $ \lambda $ on the diagonal entries, one on all of the super-diagonal entries, and zero elsewhere. We will refer to a block diagonal matrix where each block is a Jordan block as  \textit{Jordan form}. 
\end{definition}

In what follows, we will use the following notation throughout the paper.
\begin{enumerate}
	\item 
	$ \mathrm{J}(m,\lambda  ): = \lambda \mathrm{I}_m  +  \mathrm{N}(m,\lambda)$,  where $ \mathrm{N}(m, \lambda)$ is the nilpotent part of \textit{Jordan  block }$ \mathrm{J}(m, \lambda )$.
	\item
	$  \mathrm{J}(\d(n),\,\lambda): = \lambda \mathrm{I}_n  +  \mathrm{N}(\d(n),\,\lambda) =
	\oplus_{ j=1}^{s}  \Big( \, \underset{t_{d_j}\text{-many}}{\underbrace{\mathrm{J}(d_j,\lambda ) \oplus \cdots \oplus \mathrm{J}(d_j,\lambda )}} \Big) $,  where $ \d(n)=[d_1^{t_{d_1}}, \dots , d_s^{t_{d_s}}] \in \PC(n)$.   In some places, we use the notation $\d_{\lambda}$ instead of $ \d(n)$. 
	
	\item $\lambda \mathrm{I}_n $ and $ \mathrm{N}(\d(n),\,\lambda)$  are the semi-simple part and the nilpotent part of the Jordan form $\mathrm{J}(\d(n),\,\lambda)$.  For given $\lambda \in \F $ and a partition of $n \in \N$, there is a unique Jordan form $ \mathrm{J}(\d(n),\,\lambda)$.
\end{enumerate}

\subsection{Preliminary results}  Here we will include some known results.
First,  we state a basic lemma without proof.
\begin{lemma}[cf.~{\cite[  Proposition 4.2,  Theorems 4.2, 4.3]{FZ}}]\label{lem-prop-adj}
	Let $A, B  \in  \mathrm{M}(n,\H)$.  Then
	\begin{enumerate}
		\item   $\Phi({AB}) \,=\,  \Phi({A}) \Phi({B}), \,\, \Phi (A+B) \,=\,  \Phi(A) +  \Phi(B), \,\,  \Phi(\overline{A}^T) \,=\, \overline{\Phi({A})}^T$ and \\
		$  \,\Phi(A^{-1}) \,
		= \,(\Phi(A))^{-1} $ if $A^{-1} $ exists. 
		\item $ \det_{\H}(A)  \in \R \ \hbox{and} \  \det_{\H}(A) \geq 0 $ for all $A  \in  \mathrm{M}(n,\H)$.  
		\item $ A$ is invertible if and only if $\Phi(A)$ is invertible. Also,  $A  \in  \mathrm{SL}(n,\H) $ if and only if  $ \Phi(A) \in \mathrm{SL}(2n,\C)$.
	\end{enumerate}
\end{lemma}

Next, we recall the well-known Jordan canonical form in $\mathrm{M}(n,\mathbb{F})$.
\begin{theorem}
	[{cf.~ \cite[Theorem 15.1.1,  Theorem 5.5.3]{rodman}}] \label{lem-Jordan-sl(n,C)} 
	For every $X \in  \mathrm{M}(n,\F)$, there is an invertible matrix $g \in  \mathrm{GL}(n,\F)$ such that 
	\begin{equation}\label{equ-Jordan-M(n,F)}
		gXg^{-1} =  \mathrm{J}(\d(m_{\lambda_1}),\,\lambda_1) \oplus  \cdots \oplus  \mathrm{J}(\d(m_{\lambda_p}),\,\lambda_p),
	\end{equation}
	where	
	\begin{itemize}
		\item[(\textit{i})]  $ \lambda_1,  \dots,  \lambda_p  $ are distinct complex numbers when  $\F= \C$,   
		\item[(\textit{ii})]  $ \lambda_1,  \dots,  \lambda_p  $ are distinct complex numbers  with non-negative imaginary parts when $\F= \H$.  
	\end{itemize}
	The form  \eqref{equ-Jordan-M(n,F)} is uniquely determined by $X$ up to a permutation of   Jordan blocks.
\end{theorem}

The following result is also well-known, which classifies strongly ${\rm Ad} _{{\SL}(n,\C)}$-real  semi-simple element in $\mathfrak{sl}(n,\C)$.
\begin{proposition}
	\label{lem-semi-strongly-real-sl(n,C)} 
	An ${\rm Ad} _{{\SL}(n,\C)}$-real  semi-simple element in $\mathfrak{sl}(n,\C)$ is strongly ${\rm Ad} _{{\SL}(n,\C)}$-real if and only if either $0$ is an eigenvalue or $n \not\equiv 2 \pmod 4 .$	
\end{proposition}
\begin{proof}
	The proof follows from the fact that any semi-simple element in $ \s\l(n,\C) $ is conjugate to a diagonal matrix.
\end{proof}

Let us recall the following results from \cite{GM}, which classifies strongly ${\rm Ad} _{{\SL}(n,\F)}$-real  nilpotent element in $\mathfrak{sl}(n,\F)$, where $\F =\C$ or $\H$.

\begin{theorem}[{cf.~\cite[Theorem 5.3]{GM}}]\label{lem-nil-strongly-real-sl(n,H)}  
	Every nilpotent element in $ \mathfrak{sl}(n,\H) $ is strongly $ {\rm Ad} _{{\SL}(n,\H)}$-real. 
\end{theorem}

Recall that any nilpotent element $X\in\mathfrak{sl}(n,\C) $ has a unique Jordan canonical form $ \mathrm{N}(\d(n),\,0)$, where $ \d(n) $ is the corresponding partition of $ n $. This is the correspondence between  $X$ and the partition $\d(n)$ of $n$.

\begin{theorem}[{cf.~\cite[Theorem 4.6]{GM}}]\label{lem-nil-strongly-real-sl(n,C)}
	Let $ X \in \mathfrak{sl}(n,\C) $ be a nilpotent element,  and $ {\d}(n) \in \mathcal{P}(n) $ be the corresponding partition.  Then $X$ is strongly $ {\rm Ad}_{{\SL}(n,\C)} $-real if and only if $ {\d} (n)\not\in  \widetilde{ \mathcal{P}_e}(n)$, where $  \widetilde{ \mathcal{P}_e}(n) $ is as  defined in \eqref{definition-P-e-tilde}.
\end{theorem}

\section{Some Useful Lemmas}\label{sec-lemma}
The first result gives a relation between   $ {\rm Ad}_{{\rm SL }(n,\F)}$-real and  $ {\rm Ad}_{{\rm GL }(n,\F)}$-real elements in $ \s\l(n,\F) $,  where $\mathbb{F} = \C$ or $ \H$. 
\begin{lemma}\label{equi-SL-GL}
	An element $X$ of $\mathfrak{sl}(n,\F)$ is $ {\rm Ad}_{{\rm SL }(n,\F)} $-real  if and only if it is $ {\rm Ad}_{{\rm GL }(n,\F)} $-real.
\end{lemma}

\begin{proof}
	Let $gXg^{-1}=-X$ hold for some $g \in \mathrm{GL}(n,\F)$. Now if $\det (g) =1$ then we are done. Suppose that $\det (g) \ne1$. Recall that in the case $ \F=\H $, $ \det(g) \in \R$ and $ \det(g)>0 $; see \lemref{lem-prop-adj}(2). Let $ \alpha:= \frac{1}{(\det(g))^{1/n}}$. Then $ \alpha{\rm I}_n$ lies in the center of ${\GL}(n,\F)$.  Now 	
	$g(\alpha{\rm I}_n)X  (\alpha{\rm I}_n)^{-1} g^{-1}\,= \,-X$, and $ \det (g(\alpha{\rm I}_n))=1$. Hence, the proof follows.
\end{proof}

In the following example, we will describe a conjugating element $ g $ for an $ {\rm Ad}_{{\rm GL }(n,\C)} $-real nilpotent element $ X $, cf.  \cite[Example 3.11]{GM}. 
\begin{example}
	Let $ X\in \s\l(8,\C) $ be a nilpotent element which corresponds to the partition $ [4,2^2] $ of $ 8 $. There exists $ v_1,v_2,v_3\in\C^8 $ such that the vector space $ \C^8 $ has a basis of the form $ \{v_1, Xv_1, X^2v_1, X^3v_1,\, v_2, Xv_2,\,v_3, Xv_3 \} $. We now re-ordered the above basis as follows:
	\begin{align}\label{basis-B-exmp}
		\BC:=\{ X^3v_1, Xv_2,  Xv_3, X^2v_1, v_2, v_3, Xv_1, v_1  \} = (u_1,\dots, u_8)\, , \,({\rm say}).
	\end{align}
	If  $g\in {\GL}(\C^8)$ such that  $-X=g X g^{-1}$, then $ g$ is completely determined by its action on $ v_i $ for $ i=1,2,3 $. Write $ g (v_r)\,=\, \sum_{i,\,l} c^l_{i\,r} X^lv_i$ for $ r=1,2,3 $, where $ c^l_{i\,l} \in \C$. Then
	\begin{align*}
		&~ g (v_1)\,=\,  c^3_{1\,1} X^3v_1 \,  + \, c^1_{2\,1}Xv_2 \, + \,  c^1_{3\,1} Xv_3 \,  + \,  c^2_{1\,1} X^2v_1 \, + \, c^0_{2\,1} v_2 \, +\,  c^0_{3\,1} v_3 \, + \,  c^1_{1\,1} X v_1\, + \, c^0_{1\,1} v_1\,,\\
		& ~ g (v_2)\,=\,  c^3_{1\,2} X^3v_1 \,  + \, c^1_{2\,2}Xv_2 \, + \,  c^1_{3\,2} Xv_3 \,  + \,  c^2_{1\,2} X^2v_1 \, + \, c^0_{2\,2} v_2 \, +\,  c^0_{3\,2} v_3 \, + \,  \magenta{c^1_{1\,2}} X v_1\, + \, \magenta{c^0_{1\,2}} v_1\,,\\
		& ~ g (v_3)\,=\,  c^3_{1\,3} X^3v_1 \,  + \, c^1_{2\,3}Xv_2 \, + \,  c^1_{3\,3} Xv_3 \,  + \,  c^2_{1\,3} X^2v_1 \, + \, c^0_{2\,3} v_2 \, +\,  c^0_{3\,3} v_3 \, + \,  \magenta{c^1_{1\,3}} X v_1\, + \, \magenta{c^0_{1\,3}} v_1\,.
	\end{align*}
	Using the relations $-X=g X g^{-1}$, and $ X^4v_1=0,\,X^2v_i=0,\, i=2,3$; it follows that 
	\begin{enumerate}
		\item $~ g (X v_1)\,= - X g (v_1) \, = - (\, c^2_{1\,1} X^3v_1 \, + \, c^0_{2\,1} Xv_2 \, +\,  c^0_{3\,1} Xv_3 \, + \,  c^1_{1\,1} X^2 v_1\, + \, c^0_{1\,1}X v_1)$.
		\item $~ g (X v_2)\,= - X g (v_2) \, = - (\, c^2_{1\,2} X^3v_1 \, + \, c^0_{2\,2} Xv_2 \, +\,  c^0_{3\,2} Xv_3 \, + \,  c^1_{1\,2} X^2 v_1\, + \, c^0_{1\,2}X v_1)$.
		\item $~ g (X v_3)\,= - X g (v_3) \, = - (\, c^2_{1\,3} X^3v_1 \, + \, c^0_{2\,3} Xv_2 \, +\,  c^0_{3\,3} Xv_3 \, + \,  c^1_{1\,3} X^2 v_1\, + \, c^0_{1\,3}X v_1)$.
		
		\item $~ g (X^2 v_1)\,= \,   c^1_{1\,1} X^3 v_1\, + \, c^0_{1\,1}X^2 v_1$.
		\item $~ g (X^3 v_1)\,=\, -c^0_{1\,1} X^3v_1$.
	\end{enumerate}
	Since $g (X^{2} v_2) = g (X^{2} v_3) =0$,  we have $ c^1_{1\,2}  = c^0_{1\,2} =  c^1_{1\,3} = c^0_{1\,3} =0$. 
	Observe that 
	$$
	g\big( \text{Span}_\C\{u_1,\dots, u_i\} \big) \, \subseteq  \text{ Span}_\C\{u_1,\dots, u_i\}\,, \quad {\rm for\, } \  i=1, \dots 8 \,;
	$$
	where the ordered basis $ (u_1,\dots,u_8) $ is as in \eqref{basis-B-exmp}. 
	In particular, the matrix $ [g]_\BC $ has the following form:
	$$
	[g]_\BC\,=\,
	\begin{pmatrix}
		\blue{-c^0_{1\,1}} & -c^2_{1\,2}  & -c^2_{1\,3} &   c^1_{1\,1} &   c^3_{1\,2} &    c^3_{1\,3} & -c^2_{1\,1}  & c^3_{1\,1}      
		\\
		& \violet{ -c^0_{2\,2}} & \violet{  -c^0_{2\,3}} & 0 &  c^1_{2\,2}  & c^1_{2\,3}  & -c^0_{2\,1}  & c^1_{2\,1}  
		\\
		&\violet{  -c^0_{3\,2}} & \violet{ -c^0_{3\,3} }& 0 & c^1_{3\,2}  & c^1_{3\,3}  & - c^0_{3\,1} & c^1_{3\,1} 
		\\  
		&   & &  \blue{c^0_{1\,1}}  & c^2_{1\,2} & c^2_{1\,3} &- c^1_{1\,1} & c^2_{1\,1}   
		\\  
		& &  &   &\violet{  c^0_{2\,2}}  & \violet{ c^0_{2\,3}} & 0 & c^0_{2\,1} 
		\\
		&  &   &   &\violet{  c^0_{3\,2}} & \violet{ c^0_{3\,3}} &0 &c^0_{3\,1}
		\\
		& &  &   &   & & \blue{-c^0_{1\,1}}   & c^1_{1\,1}   
		\\
		& &  &   &   &  & &  \blue{c^0_{1\,1}} 	  
	\end{pmatrix} \,.$$
	\qed
\end{example}

In the proof of the next result,  we generalize the above description of the conjugating elements in \eqref{matrix-g-in-basis}; see also  \cite[Section 3.5]{GM}.
The following result will be used in Lemma \ref{lem-partition-sl(n,H)}.
\begin{lemma} \label{lem-nilpotent-reverser}
	Let $X \in \mathfrak{sl}(n,\C)$ be a nilpotent element and  $\d(n)= [d_1^{t_{d_1}}, \dots,d_s^{t_{d_{s}}} ]\in \PC(n) $  be the associated partition for the nilpotent element $ X $.   Let $g\in \GL(n,\C) $ be such that $g$ satisfies the following properties:
	\begin{enumerate}
		\item $ gX\,=\,-Xg $,
		\item $ g\bar{g}\, =\, -{\rm I}_n  $.
	\end{enumerate}
	Then $ t_{d_{j}} $ is even for all $ j=1, \dots , s $.
\end{lemma}

\begin{proof}
	Given a nilpotent element $ X$ in $\mathfrak{sl}(n,\C)$, the underlying vector space $ \C^n $ has a basis of the form $ \{X^lv_i^{d_j}\, \mid\, 1\leq j\leq s,\, 1\leq i\leq t_{d_j},\, 0\leq l<d_j\} $,  where $ [d_1^{t_{d_1}}, \dots,d_s^{t_{d_s}} ]\in \PC(n) $ is the associated partition for $ X $, \cite[Chapter IV (1.6), p. E-83]{SS}.	Motivating by an ordering as done in \cite[(4.4)]{BCM}, we will consider here the following ordering as in \cite[Section 3.5]{GM}.
	Set 
	$\BC^l (d_j) \,:=\, (X^l v_1^{d_j},\, \ldots ,\,X^l v^{d_j}_{t_j})$. 
	Define
	\begin{equation}\label{old-ordered-basis}
		\BC(j) \,:=\, \BC^{d_1-j} (d_1) \vee    \cdots \vee  \BC^{d_s-j} (d_s)  ,\ \text{ and }\  \BC\,:=\, \BC(1) \vee \cdots \vee  \BC(d_1)\, .
	\end{equation}
	Then  $\BC$ is an ordered basis of  $ \C^n $. From the definition of $ \BC(1), X|_{ \BC(1)}=0$. Using the relation $(-X)^r=g X^r g^{-1}$, $r\in \N$, it follows that 
	\begin{enumerate}
		\item $g\big({\rm Span}_{\C}\{\BC^{d_1-1}(d_1)\vee \cdots\vee \BC^{d_j-1}(d_j) \}\big) \,\subseteq \, {\rm Span}_\C\{\BC^{d_1-1}(d_1)\vee \cdots\vee \BC^{d_j-1}(d_j) \} $\\  for $ j=1,\ldots, s $.
		\item $g\big( \text{Span}_\C\{\BC(1)\vee \cdots \vee \BC(m) \}\big) \,\subseteq \, \text{Span}_\C\{\BC(1)\vee \cdots \vee \BC(m) \}$ for $m=1,\ldots , d_1$. 
	\end{enumerate}
	Also,  $g$ is determined uniquely by its action on $ v_i^{d_j} $ for $  1\leq j\leq s,\, 1\leq i\leq t_{d_j} $. 
	Note that the matrix  $[g]_\BC$ is a block upper triangular matrix with $ (d_1+\cdots +d_s) $-many diagonal blocks.
	The matrix  $[g]_\BC$ is of the following form:
	\begin{align}\label{matrix-g-in-basis}
		[g]_\BC\,=\,
		\begin{pmatrix}
			g_{1\,1} &  \cdots &  \ \  \cdots  &g_{1\,n}\\
			&\ddots &   &\vdots \\
			&&   g_{r\,r} \  \cdots & g_{r\,n}\\     
			& &\qquad  \ddots   &\vdots \\
			& && g_{n\,n}
		\end{pmatrix}\,,
	\end{align}  
	where the order of  the first $s$-many diagonal blocks $g_{jj}$ is
	$ t_{d_j} \times t_{d_j} $,  for $j=1,\ldots , s$. 
	The condition $ g\overline{g} =-{\rm I}_n$ implies that $ g_{jj}$, the diagonal block of $[g]_\BC$,  satisfies 
	$$
	g_{jj}\overline{ g}_{jj}\,=\,-{\rm I}_{t_{d_j}} \,,\quad  j=1,\dots,s \,.
	$$
	By considering the determinant in the above relation, it follows that $ \det (	g_{jj}) \overline{\det( g_{jj})} \,=\,|\det( g_{jj}) |^2\,=\, (-1)^{{t_{d_j}}} $. 
	Thus
	$ {t_{d_j}}$, the order of $ g_{jj} $, is even for all $ j=1, \dots , s $. This completes the proof.
\end{proof}

The following lemma is an infinitesimal version of {\cite[Lemma 2.2.1]{ST}}.
\begin{lemma}\label{lem-real-lie-algebra}
	Let $\mathfrak{g} $ be a Lie algebra of a connected Lie group $G$. Let $X \in \mathfrak{g} $ and $ X = X_s + X_n $ be the Jordan decomposition of $X$.  Then $X$ is ${\rm Ad}_G$-real  if and only if both the following conditions hold:
	\begin{enumerate}
		\item  $X_s$ is ${\rm Ad}_G$-real.
		\item  $- X_n$ and $ \sigma X_n \sigma^{-1} $ are conjugate in $ \mathrm{Z}_G(X_s)$ 
		for some $\sigma \in G $ such that $ \sigma X_s \sigma^{-1} = - X_s $.
	\end{enumerate}
\end{lemma}

\begin{proof} First assume that $X$ is ${\rm Ad}_G$-real. Let $g \in G$ be such that $ - X_s - X_n = - X = g X g^{-1}  = g (X_s + X_n ) g^{-1}  = g X_s g^{-1} + g X_n g^{-1} $.   From the uniqueness of the Jordan decomposition of $X,$ we get $ - X_s = g X_s g^{-1} \ \hbox{and} \ - X_n = g X_n g^{-1}. $
	
	Conversely,  let $  - X_s =  \sigma X_s \sigma^{-1} $ and $  - X_n =  h \sigma X_n \sigma^{-1} h^{-1} $ for some $\sigma \in G $ and $ h \in  \mathrm{Z}_G(X_s) $. Then it follows that 
	$  - X= (h\sigma) (X_s + X_n) (h\sigma)^{-1} = (h\sigma) X (h\sigma)^{-1}. $  Hence, $X$ is ${\rm Ad}_G$-real.    
\end{proof}  

In the following lemma, we get the centralizer of semi-simple elements in $\mathfrak{sl}(n,\C)$.

\begin{lemma}\label{lem-center-SL(n,C)}
	Let $ A  \in \mathfrak{sl}(n,\C) $ be a non-zero semi-simple element. For $ 1 \leq i \leq r $, let $ \lambda_i$ denote the distinct eigenvalues of $ A$ with the multiplicity  $ p_i$. Then 
	$$ \mathrm{Z}_{\mathrm{GL}(n,\C)}(A) := \bigg\{ g \in \mathrm{GL}(n,\C) \,\mid\, g Ag^{-1}  = A \bigg\}\, \simeq \, \bigoplus_{i=1}^{r} \mathrm{GL}(p_i,\C) \,.$$ 
\end{lemma}

\begin{proof} Since $A$ is a semi-simple element in $\s\l(n,\C)$, up to conjugacy  $ A$ is of the form $\oplus_{ i=1}^{r}  \lambda_i \mathrm{I}_{p_i} $. The proof now follows from straightforward calculations.
\end{proof}

The following lemma  characterizes conjugating   elements corresponding to   $ {\rm Ad} _{{\SL}(n,\C)}$-real semi-simple elements of $\mathfrak{sl}(n,\C)$. Such description has been used in proving  Theorem \ref{thm-real-sl(n,C)}.
\begin{lemma}\label{lem-conj-semi-sl(n,C)} 
	Let   $X_s \in \mathfrak{sl}(n,\C)$ be an element  such that 
	$$
	X_s \,=\, \mathrm{O}_{p_o \times p_o} \, \bigoplus\, \big( \oplus_{ i=1}^{r} ( \lambda_i  \mathrm{I}_{p_i} ) \big)\, \bigoplus\, \big( \oplus_{ i=1}^{r}  (-\lambda_i  \mathrm{I}_{p_i} )\big),
	$$
	where $ p_o \in \N \cup \{0\}$,  $ p_i \in \N $ and $\lambda_i \neq 0 $ for all $1 \le i \le r $  such that $ \lambda_j \neq \pm \lambda_k \ \hbox{for all} \  j \neq k$.  Suppose that $ \sigma X_s \sigma^{-1} = - X_s $ for some  $ \sigma \in \mathrm{SL}(n,\C)$.  Then  $\sigma $ has the following form: 
	\begin{align}\label{def-sigma-elt}
		\sigma =  \alpha_{p_o \times p_o} \bigoplus  \begin{pmatrix}
			&  (\oplus_{ i=1}^{r}  f_{p_i}) \\
			(\oplus_{ i=1}^{r}  g_{p_i}) &  
		\end{pmatrix}, 
	\end{align}
	where  $ \sum_{i=1}^{r} p_i = p$ and $ f_{p_i} ,  g_{p_i}   \in  \mathrm{GL}(p_i,\C),  \alpha_{p_o \times p_o} \in \mathrm{GL}(p_o,\C)$.
\end{lemma}

\begin{proof}
	Let $ A :=  \oplus_{ i=1}^{r} ( \lambda_i  \mathrm{I}_{p_i} ) \in  \mathrm{GL}(p,\C). $ Then  $ X_s = \mathrm{O}_{p_o \times p_o} \oplus A \oplus (- A). $
	Let  $$\sigma :=  \begin{pmatrix}
		L_{p _o \times p_o}   &  M_{p_o \times p} &  N_{p_o \times p}    \\
		P_{ p \times p_o} &  Q_{p \times p}  &  R_{p \times p}\\ 
		S_{p \times p_o} &  T_{p \times p}  &  U_{p \times p}\\ 
	\end{pmatrix}  \in \mathrm{SL}(n,\C)$$
	be  such that $ \sigma X_s \sigma^{-1} = - X_s $. 
	Then the following conditions hold on sub-matrices of $ \sigma $:
	\begin{enumerate}
		\item  $ MA = NA =   \mathrm{O}_{p _o \times p} $ and $ AP = AS = \mathrm{O}_{p \times p_o}.$ Since $ A \in \mathrm{GL}(p,\C)$,  it follows that $ M = N = \mathrm{O}_{p _o \times p} \ \hbox{and} \,\, P = S = \mathrm{O}_{p \times p_o}. $
		
		\item $ QA = - AQ \ \hbox{and}  \ AU = - UA $, where $ A =  \oplus_{ i=1}^{r} ( \lambda_i  \mathrm{I}_{p_i} )$,   $\lambda_i \neq 0 $ and  $ \lambda_j \neq \pm \lambda_k \ \hbox{for all} \  j \neq k.$ Therefore, $ Q = U =  \mathrm{O}_{p \times p}.$
		
		\item $ AR = RA \ \hbox{and}  \ AT = TA $, where  $ A =  \oplus_{ i=1}^{r} ( \lambda_i  \mathrm{I}_{p_i} )$,   $\lambda_i \neq 0 $ and $ \lambda_j \neq \pm \lambda_k \ \hbox{for all} \  j \neq k.$ By using Lemma \ref{lem-center-SL(n,C)}, we get   $ R, T  \in \mathrm{Z}_{\mathrm{GL}(n,\C)}(A) $ such that $ R = \oplus_{ i=1}^{r}  f_{p_i},  T = \oplus_{ i=1}^{r}  g_{p_i} $,  where $ f_{p_i} ,  g_{p_i}   \in  \mathrm{GL}(p_i,\C)$ for all $1 \leq i \leq r$.  
	\end{enumerate} 
	The proof of lemma now follows from the above observations.  
\end{proof}

Next lemma gives us a relationship between $ {\rm Ad}_{{\rm SL}(n,\H)} $-reality in $\mathfrak{sl}(n,\H) $ and $ {\rm Ad}_{{\rm SL}(2n,\C)} $-reality in $ \mathfrak{sl}(2n,\C) $.
\begin{lemma}\label{lem-real-H- to-real-C}
	If $X \in \mathfrak{sl}(n,\H) $ is $ {\rm Ad}_{{\rm SL}(n,\H)} $-real (resp.  strongly $ {\rm Ad}_{{\rm SL}(n,\H)} $-real),  then $\Phi(X) \in  \mathfrak{sl}(2n,\C)$ is $ {\rm Ad}_{{\rm SL}(2n,\C)} $-real (resp.  strongly   $ {\rm Ad}_{{\rm SL}(2n,\C)} $-real).
\end{lemma}
\begin{proof}
	The proof follows from the fact that the map $ \Phi $ in \eqref{definition-Phi-embedding} is a homomorphism.	
\end{proof}

The converse of \lemref{lem-real-H- to-real-C} is not true; see Example \ref{ex1}.

\section{Proof of the main results}\label{sec-pf-thm}

\subsection{Classification of real and  strongly real elements in  $\mathfrak{sl}(n,\C) $}

Proof of \thmref{thm-real} is divided into two parts according to $\F=\C$ or $\H$. 
Here we will prove \thmref{thm-real} for the case $ \F=\C$.	The case $ \F=\H $ will be dealt in \secref{sec-sl-n-H}.

\begin{theorem}\label{thm-real-sl(n,C)}
	Let $X \in \mathfrak{sl}(n,\C)$ be a non-zero arbitrary element. Then $X$ is $ {\rm Ad}_{{\rm SL }(n,\C)} $-real  if and only if both the following conditions hold:
	\begin{enumerate}
		\item \label{cond-1-real-sl(n,C)-thr}  
		If $\lambda$ is an eigenvalue of $X$, then $ - \lambda $ is also an eigenvalue of $X$ with the {\it same} multiplicity.
		\item \label{cond-2-real-sl(n,C)-thr}  
		Let $ m $ be the multiplicity of  the eigenvalues $ \lambda $ and $ -\lambda $.
		Let $ \d_\lambda $  and $ \d_{-\lambda} $ be  the partitions  of $ m $ associated to the eigenvalues $ \lambda $ and $ -\lambda $, respectively. Then 
		$$ \d_\lambda\,=\, \d_{-\lambda}  \,.$$
	\end{enumerate}
\end{theorem}

\begin{proof}	
	First we assume that  $X\in \s\l(n,\C)$ satisfies conditions  \eqref{cond-1-real-sl(n,C)-thr}  and  (\ref{cond-2-real-sl(n,C)-thr}) of \thmref{thm-real-sl(n,C)}.  	
	Then, the Jordan decomposition of $X$  will be of the following form: 
	\begin{align}\label{Jordan-decomp-real-elt}
		X = X_s + X_n 
	\end{align} 
	such that  
	\begin{itemize}
		\item[(\textit{i})]
		$  X_s = \mathrm{O}_{p_o \times p_o} \bigoplus \Big(\oplus_{ i=1}^{r} (\lambda_i  \mathrm{I}_{p_i} )\Big) \bigoplus \   \Big(\oplus_{ i=1}^{r}  (- \lambda_i  \mathrm{I}_{p_i})\Big) $, \vspace{.2cm}
		\item [(\textit{ii})] 
		$X_n = \mathrm{N}(\d_0, 0) \bigoplus \Big( \oplus_{ i=1}^{r}  \mathrm{N}(\d_{\lambda_i},  \lambda_i)\Big) \bigoplus \Big( \oplus_{ i=1}^{r}  \mathrm{N}( \d_{{-\lambda_i}}, - \lambda_i)\Big),$
		where $p= \sum_{i=1}^{r} p_i ,$  $   p_o +2p = n $ and $\lambda_i \neq 0$ for all $ i $,  $ \lambda_j \neq \pm \lambda_k \ \hbox{for all} \  j \neq k.$ 
	\end{itemize}
	
	Let  $$\sigma := \mathrm{I}_{p_o} \bigoplus  \begin{pmatrix}
		&  - \mathrm{I}_{p}\\
		\mathrm{I}_{p} &   \\ 
	\end{pmatrix}\,  \,\,{\rm and }\quad \tau := \textnormal{diag} (z,  -z,  z,  -z, \dots , (-1)^{n+1}z \ )_{n \times n }, $$ 
	where $z \in \C \setminus \{0\} $ such that $ \tau \in \mathrm{SL}(n,\C).$ Note that $ \sigma \in \mathrm{SL}(n,\C) $ such that $ \sigma X_s \sigma^{-1}$  $ = - X_s $ and $ \sigma X_n \sigma^{-1} = X_n.$ Similarly,
	$ \tau X_s \tau^{-1} = X_s $ and $ \tau X_n \tau^{-1} = - X_n $.
	Set $$  g: = \tau \sigma \,.$$
	Then $gXg^{-1} = (\tau \sigma)( X_s + X_n )(\tau \sigma )^{-1} = -X$, and $ \det g=1 $.
	Hence, $X$ is $ {\rm Ad}_{{\rm SL }(n,\C)} $-real.
	
	Next, we assume that $ X $ is $ {\rm Ad}_{{\rm SL }(n,\C)} $-real. 
	Then $X_s$ is so; see Lemma \ref{lem-real-lie-algebra} (1).
	Now condition (\ref{cond-1-real-sl(n,C)-thr}) of  the Theorem \ref{thm-real-sl(n,C)} follows.
	
	Moreover, using Lemma \ref{lem-real-lie-algebra}(2), we have	$- X_n$ and $ \sigma X_n \sigma^{-1} $ are conjugate in $ \mathrm{Z}_{\mathrm{SL}(n,\C)}(X_s)$ 
	for some $\sigma \in \mathrm{SL}(n,\C) $ such that $ \sigma X_s \sigma^{-1} = - X_s $.
	Then $\sigma$ is of the form \eqref{def-sigma-elt}. 
	In view of 	
	Lemma \ref{lem-center-SL(n,C)},  it follows  that there exists $ h =  \oplus_{ j=0}^{2r}  h_{p_j} \in \mathrm{Z}_{\mathrm{SL}(n,\C)}(X_s) $ such that  $  hX_n h^{-1} = - \sigma X_n \sigma^{-1}$,   where $ h_{p_j} \in  \mathrm{GL}(p_j,\C) $ for all $ j \in \{ 0, 1,  2,  \dots , 2r \}.$
	Now the relation $  hX_n h^{-1} = - \sigma X_n \sigma^{-1}$ implies that  $h_{p_i}  \mathrm{N} (\d_{\lambda_i}, \lambda_i) h_{p_i}^{-1} = -  f_{p_i}  \mathrm{N}(\d_{-\lambda_i}, - \lambda_i) f_{p_i}^{-1} $ for  all $ i \in \{ 1,   \dots ,  r \}$,  where $ f_{p_i} $'s are as in \eqref{def-sigma-elt}. Therefore, the conclusion follows from the Jordan canonical form over $\C$; see \thmref{lem-Jordan-sl(n,C)}. 
\end{proof}

Next, we will prove \thmref{thm-strongly-real-sl(n,C)}.  Recall the definition of $ \widetilde{ \mathcal{P}_e}(n)$ as in (\ref{definition-P-e-tilde}).

\noindent {\bf Proof of \thmref{thm-strongly-real-sl(n,C)}.~}
First we assume that either $0$ is an eigenvalue of $X$ such that $ {\d}(p_o) \not\in  \widetilde{ \mathcal{P}_e}(p_o)$ or $n \not\equiv 2 \  (mod \ 4 \ )$.
We will divide the above situation into two cases. 

{\it Case 1.~}
When  	$n \not\equiv 2 \  (mod \ 4 \ )$ and $p_o = 0$, i.e., $n= 2p$, where $p$ is even.
Define 
$$\sigma := \begin{pmatrix}
	&   \mathrm{I}_{p}\\
	\mathrm{I}_{p} &    \\ 
\end{pmatrix}  \text{ and } \tau :=  \tau_1 \oplus \tau_1 , \, \hbox{where } \, \tau_1:=\textnormal{diag} (1,  -1,  1,  -1, \dots , (-1)^{p+1} \ )_{p \times p }.$$
Since $p$ is even, $\det(\sigma)= (-1)^p=1$.
Further, note that $ \tau\sigma = \sigma \tau$.

{\it Case 2.~}	When  $ p_o \neq 0$ and $ {\d}(p_o) \not\in  \widetilde{ \mathcal{P}_e}(p_o)$, set  $\sigma := \pm \mathrm{I}_{ 1} \oplus \mathrm{I}_{p_o - 1} \bigoplus  \begin{pmatrix}
	&  \mathrm{I}_{p}\\
	\mathrm{I}_{p} &    \\ 
\end{pmatrix}$,  where $+ \mathrm{I}_{ 1} $ or  $ -\mathrm{I}_{ 1} $ is chosen so that $ \sigma \in \mathrm{SL}(n,\C) $.
Since $ {\d}(p_o) \not\in  \widetilde{ \mathcal{P}_e}(p_o)$, using  \thmref{lem-nil-strongly-real-sl(n,C)}, we get an {\it involution} $\tau_o \in \mathrm{SL}(p_o,\C) $ such that $$ \tau_o \  \mathrm{N}(\d(p_o), 0) \  \tau_o^{-1} = -  \mathrm{N}(\d(p_o), 0). $$
Further, we can choose $\tau_o$ such that $\tau_o$ is diagonal; see {\cite[Theorem 4.6, Proposition 4.4]{GM}}.
Let  $ \tau := \tau_o  \oplus  \tau_1 \oplus  \tau_1 $ such that $ \tau_1 =  \textnormal{diag} (1,  -1,  1,  -1, \dots , (-1)^{p+1} \ )_{p \times p }$.  Note that by our choice of $\tau_o$, it follows that $ \tau  \sigma = \sigma \tau$.

Now in both the above cases,  we have observed that $ \sigma, \tau\in \mathrm{SL}(n,\C) $ are involutions such that $ \sigma X_s \sigma^{-1} = - X_s $, $ \sigma X_n \sigma^{-1} = X_n$, $ \tau X_s \tau^{-1} = X_s $, and $ \tau X_n \tau^{-1} = - X_n $. Since  $ \tau\sigma = \sigma \tau$, it follows that $ \tau \sigma \in {\SL}(n,\C)$ is an involution and  $ (\tau \sigma) X( \tau \sigma)^{-1} = -X$.  Hence, $X$ is strongly $ {\rm Ad}_{{\rm SL}(n,\C)} $-real. 

Now, we assume that  $X $ in  $\mathfrak{sl}(n,\C) $ is strongly  $ {\rm Ad}_{{\rm SL}(n,\C)} $-real.  This implies that $X_s$ and $X_n$ are strongly $ {\rm Ad}_{{\rm SL}(n,\C)} $-real.   
Using Proposition \ref{lem-semi-strongly-real-sl(n,C)}, we have either $0$ is an eigenvalue or $n \not\equiv 2 \  (mod \ 4 \ ).$  Therefore,  if $0$ is not an eigenvalue of $X$, then we are done.

Next, we assume that $0$ is an eigenvalue of $X$, i.e.,  $p_o \neq 0 \  \hbox{and}  \ p_o \in \N $.  Recall that $ {\d}(p_o) \in \mathcal{P}(p_o) $ be the corresponding partition to the eigenvalue $ 0 $.
Since $X_n \in \mathfrak{sl}(n,\C) $ is strongly real,  using  \thmref{lem-nil-strongly-real-sl(n,C)}  it follows  that  if $ {\d}(n) \in \mathcal{P}(n) $ be the corresponding partition of $X_n$, then $ {\d}(n) \not\in  \widetilde{ \mathcal{P}_e}(n).$ 	In view of  the Jordan decomposition of $X$  as given in \eqref{Jordan-decomp-real-elt},  it follows that  if $ {\d}(n) \not\in  \widetilde{ \mathcal{P}_e}(n)$ then  ${\d}(p_o) \not\in  \widetilde{ \mathcal{P}_e}(p_o)$. This completes the proof.
\qed

\subsection{Classification of  real and strongly real  elements in  $\mathfrak{sl}(n,\H) $}\label{sec-sl-n-H}

Recall that (right) eigenvalues of a matrix over $\H$ exist in similarity classes, and we can represent an eigenvalue class by their \textit{unique} complex representative; see Definition \ref{def-eigen-M(n,H)}.

The following lemma classifies  $ {\rm Ad}_{{\rm SL }(n,\H)} $-real semi-simple elements in $\mathfrak{sl}(n,\H) $.
\begin{lemma}\label{lem-real-semi-sl(n,H)}
	Let $X \in \mathfrak{sl}(n,\H) $ be a semi-simple element. Then the following statements are equivalent.
	\begin{enumerate}
		\item $X$ is $ {\rm Ad}_{{\rm SL }(n,\H)} $-real.
		\item \label{cond-real-semi-sl(n,H)} 
		If $\lambda$ is an eigenvalue of $X$ such that the real part of $\lambda$ is non-zero, then $-\lambda$ is also an eigenvalue of $X$.
	\end{enumerate} 
\end{lemma}
\begin{proof} 
	Since $X$ is $ {\rm Ad}_{{\rm SL }(n,\H)} $-real,  so $ - X =  g X g^{-1} $ for some $ g \in \mathfrak{sl}(n,\H) $.  This implies that if $\lambda$ is an eigenvalue of $X$, then $ - \lambda $ is also an eigenvalue of $X$ with the same multiplicity. The proof now follows from the fact that the eigenvalues of $X$ exist in similarity classes, and if $\lambda$ is a purely imaginary unique complex representative of an eigenvalue of $X$ then $[\lambda] =[-\lambda]$.
	
	Conversely,  let $X \in \mathfrak{sl}(n,\H) $ be a semi-simple element satisfying  (\ref{cond-real-semi-sl(n,H)}) of \lemref{lem-real-semi-sl(n,H)}.  Using \thmref{lem-Jordan-sl(n,C)}, we can write the  Jordan decomposition of $X$ as:
	\begin{equation}
		X = \Big(\oplus_{i =1}^{m} (\mu_i \mathrm{I}_{q_{i}} ) \Big) \bigoplus \, \Big(\oplus_{ j=1}^{r} (\lambda_j \mathrm{I}_{p_j} )\Big)  \bigoplus \, \Big(\oplus_{ j=1}^{r}  (- \lambda_j \mathrm{I}_{p_j})\Big) ,  
	\end{equation} where $ \mu_i,  \lambda_j  \in \C$ such that $ {\rm Re}(\mu_i) = 0,  {\rm Re}(\lambda_j) \neq 0, {\rm Im}(\mu_i)\geqslant 0,  {\rm Im}(\lambda_j)\geqslant 0$, 
	and $ \sum_{i=1}^{m} q_i = q,  \ \sum_{j=1}^{r} p_j = p,  \  q +2p = n$.
	
	Define $\sigma := (\mathrm{I}_{q}) \jb  \bigoplus 
	\begin{pmatrix}
		&  - \mathrm{I}_{p}\\
		\mathrm{I}_{p} &    \\ 
	\end{pmatrix} = \begin{pmatrix}
		(\mathrm{I}_{q}) \jb   &   &      \\
		&      &  - \mathrm{I}_{p}\\ 
		&  \mathrm{I}_{p}  &    \\ 
	\end{pmatrix} \in  \mathrm{SL}(n,\H)$. 
	To see $\sigma  \in  \mathrm{SL}(n,\H)$,  we write $\sigma  = A + B \jb $, where  $AB = BA $ such that  $A =  \begin{pmatrix}
		\mathrm{I}_{q}   &   &      \\
		&     &   \\ 
		&   &  \mathrm{O}_{ 2p \times 2p }  \\ 
	\end{pmatrix} $ and $ B = \begin{pmatrix}
		\mathrm{O}_{ q \times q }  &  &      \\
		&    &  - \mathrm{I}_{p}\\ 
		&  \mathrm{I}_{p}  &     \\ 
	\end{pmatrix} \in  \mathrm{M}(n,\C)$.  Using properties of determinants of block matrices, we get
	$$
	{\rm det_{\H}(\sigma)} = \det(\Phi(\sigma))=  \det \begin{pmatrix}
		A   &  B \\
		- \overline{B} & \overline{A}  \\ 
	\end{pmatrix} =  \det(A^2 + B^2) = \det \begin{pmatrix}
		\mathrm{I}_{q} &    &      \\
		&   - \mathrm{I}_{p}  &     \\ 
		&      & -\mathrm{I}_{p}  \\ 
	\end{pmatrix} = 1.
	$$
	Thus we get  $\sigma  \in  \mathrm{SL}(n,\H)$ such that $ \sigma X \sigma^{-1} = - X $.  
	This proves the lemma.
\end{proof}

In the following, we will prove \thmref{thm-real} for $ \F=\H $.

\begin{theorem}\label{thm-real-sl(n,H)}
	Let $X \in \mathfrak{sl}(n,\H)$.  Then $X $ is ${\rm Ad}_{{\rm SL }(n,\H)} $-real if and only if both the following conditions hold:
	\begin{enumerate}
		\item \label{cond-1-real-sl(n,H)-thr} If $ \lambda$ is an eigenvalue of $X$ such that real part of $\lambda$ is non-zero, then $-\lambda$ is also an eigenvalue of $X$ with the same multiplicity. 
		\item \label{cond-2-real-sl(n,H)-thr} Let $ m $ be the multiplicity of  the eigenvalues $ \lambda $ and $ -\lambda $. 
		Let $ \d_\lambda $  and $ \d_{-\lambda} $ be the partitions of $ m $ associated to the eigenvalues $ \lambda $ and $ -\lambda $, respectively. Then 
		$$ \d_\lambda\,=\, \d_{-\lambda}  \,.$$
	\end{enumerate} 
\end{theorem}

\begin{proof}
	First we assume that $X \in \mathfrak{sl}(n,\H) $ satisfies conditions (\ref{cond-1-real-sl(n,H)-thr}) and (\ref{cond-2-real-sl(n,H)-thr}) of \thmref{thm-real-sl(n,H)}. 
	Then the Jordan decomposition of $X$ will be of the following form:
	\begin{equation}
		X = X_s + X_n 
	\end{equation}
	such that
	\begin{itemize}
		\item [(\textit{i})] $ X_s = \Big( \oplus_{i =1}^{\ell} (\mu_i \mathrm{I}_{q_{i}} ) \Big) \bigoplus  \,  \Big( \oplus_{ j=1}^{r} (\lambda_j \mathrm{I}_{p_j} ) \Big) \, \bigoplus \, \Big( \oplus_{ j=1}^{r}  (- \lambda_j \mathrm{I}_{p_j}) \Big)$,
		\item [(\textit{ii})]$ X_n = \Big(  \oplus_{ i=1}^{\ell}   \mathrm{N}(\d_{\mu_i},  \mu_i) \Big) \, \bigoplus \Big( \oplus_{ j=1}^{r}  \mathrm{N}(\d_{\lambda_j},  \lambda_j) \Big) \bigoplus \Big( \oplus_{ j=1}^{r}  \mathrm{N}( \d_{-\lambda_j},  -\lambda_j) \Big)$, 
	\end{itemize}
	where $ \mu_i,  \lambda_j  \in \C$ such that ${\rm Re}(\mu_i) = 0,  {\rm Re}(\lambda_j) \neq 0,  {\rm Im}(\mu_i)\geqslant 0,  {\rm Im}(\lambda_j)\geqslant 0$
	and $ \sum_{i=1}^{\ell} q_i = q,  \ \sum_{j=1}^{r} p_j = p,  \  q +2p = n$. Note that $\mathrm{N}(\d_{\lambda_j},  \lambda_j) = \mathrm{N}( \d_{-\lambda_j},  -\lambda_j)$ for each $j = 1,2, \dots,r$.
	
	Set $  g := \tau \sigma$, where $\sigma$, $\tau \in \mathrm{SL}(n,\H) $ such that 
	$$\sigma = (\mathrm{I}_{q}) \jb  \oplus  \begin{pmatrix}
		&  - \mathrm{I}_{p}\\
		\mathrm{I}_{p} &  \\ 
	\end{pmatrix} \hbox{ and } \tau =   \textnormal{diag} (1,  -1,  1,  -1, \dots , (-1)^{n+1} \ )_{n \times n}.$$ 
	Then $g  \in \mathrm{SL}(n,\H)$ such that $gXg^{-1} =-X$.  Hence, $X $ is ${\rm Ad}_{{\rm SL }(n,\H)} $-real.
	
	Conversely,  let $X $ is ${\rm Ad}_{{\rm SL }(n,\H)} $-real.  In view of the Lemma \ref{lem-real-H- to-real-C}, it follows that $\Phi(X) \in \mathfrak{sl}(2n,\C) $ is ${\rm Ad}_{{\rm SL }(2n,\C)} $-real.  Now the proof  follows  from Theorem \ref{thm-real-sl(n,C)}, and from the fact that if $\lambda$ is an eigenvalue of $X \in \mathfrak{sl}(n,\H) $ with multiplicity (say $m$),  then $\lambda$ and $\bar{\lambda}$ are also  eigenvalues of $\Phi(X) \in \mathfrak{sl}(2n,\C) $ such that both have the  same multiplicity $m$.
\end{proof}

There are elements in $ \mathfrak{sl}(n,\H) $ which are ${\rm Ad}_{{\rm SL }(n,\H)} $-real but not strongly ${\rm Ad}_{{\rm SL }(n,\H)} $-real, e.g.,  $X = [\ib] \in \mathfrak{sl}(n,\H) $.  Now we will classify strongly ${\rm Ad}_{{\rm SL }(n,\H)} $-real semi-simple elements in $\mathfrak{sl}(n,\H) $.

\begin{lemma}\label{lem-strongly-real-semi-sl(n,H)}
	Let $X \in \mathfrak{sl}(n,\H) $ be a ${\rm Ad}_{{\rm SL }(n,\H)} $-real semi-simple element.  Then $X$ is strongly ${\rm Ad}_{{\rm SL }(n,\H)} $-real if and only if every non-zero eigenvalue class of $X$ with purely imaginary complex representative has \textit{even multiplicity}.
\end{lemma}
\begin{proof} 
	Let $X \in \mathfrak{sl}(n,\H) $ be a ${\rm Ad}_{{\rm SL }(n,\H)} $-real semi-simple element such that all non-zero purely imaginary eigenvalues of $X$ has even multiplicity.   Then by using \thmref{lem-Jordan-sl(n,C)} and Lemma \ref{lem-real-semi-sl(n,H)}, we can write  the Jordan decomposition of $X$ as :
	\begin{equation}
		X = \mathrm{O}_{p_o \times p_o} \,  \bigoplus \Big( \oplus_{i =1}^{\ell} (\mu_i \mathrm{I}_{2q_{i}} ) \Big) \bigoplus  \,  \Big( \oplus_{ j=1}^{r} (\lambda_j \mathrm{I}_{p_j} ) \Big) \, \bigoplus \, \Big( \oplus_{ j=1}^{r}  (- \lambda_j \mathrm{I}_{p_j}) \Big),  
	\end{equation} where $ \mu_i,  \lambda_j  \in \C \setminus \{0\}$ such that $ {\rm Re}(\mu_i) = 0,  {\rm Re}(\lambda_j) \neq 0,  {\rm Im}(\mu_i) > 0,  {\rm Im}(\lambda_j)\geqslant 0$,
	and $ \sum_{i=1}^{\ell} q_i = q,  \ \sum_{j=1}^{r} p_j = p,  \  p_o +2q +2p = n$.
	
	Let $\sigma = \mathrm{I}_{p_o} \bigoplus \Big( \, \oplus_{i =1}^{\ell}  \textnormal{diag} \Big(   \, \underset{q_{i}\text{-many}}{\underbrace{ \begin{pmatrix}
				0 & \jb\\
				- \jb & 0 \\ 
			\end{pmatrix},  \cdots,  \begin{pmatrix}
				0 & \jb\\
				- \jb & 0 \\ 
	\end{pmatrix}}} \Big)_{2q_i \times 2q_i} \Big) \,  \bigoplus \, \begin{pmatrix}
		&   \mathrm{I}_{p}\\
		\mathrm{I}_{p} &    \\ 
	\end{pmatrix}$.  
	Note that $\sigma  \in  \mathrm{GL}(n,\H)$ is an involution and satisfies $ \sigma X \sigma^{-1} = - X $.  In view of the Lemma \ref{lem-prop-adj}(2), we have $\sigma  \in  \mathrm{SL}(n,\H)$.  Hence, $X \in \mathfrak{sl}(n,\H) $ is strongly ${\rm Ad}_{{\rm SL }(n,\H)} $-real.
	
	Conversely,  let $X \in \mathfrak{sl}(n,\H) $ be a strongly ${\rm Ad}_{{\rm SL }(n,\H)} $-real semi-simple element.  Then there exists an involution
	$\sigma = [\sigma_{\ell,t}]_{1\leq \ell, t \leq n} \in  \mathrm{SL}(n,\H)$ such that $ \sigma X \sigma^{-1} = - X $.  Assume that $X$ has an eigenvalue class of odd multiplicity $m$ with representative ${a \ib }$, where ${a} \in \R, a > 0$.   Now there are two possible cases:
	
	{\it Case 1.~} Let  $X$ has exactly one eigenvalue class $[{a \ib }]$, i.e., $m=n$. Then up to conjugacy, we can assume that
	$$X=[ x_{\ell,t}]_{1\leq\ell,t \leq n}  = ({a \ib }) {\rm I}_{n}.$$
	The relation $ \sigma X = - X \sigma $ implies that $( \sigma_{\ell,t}) ({a \ib }) = (- {a \ib }) ( \sigma_{\ell,t}) ~ \text{if  } 1\leq \ell,t \leq n$. Thus $ \sigma_{\ell,t} =  ( a_{\ell,t}) \jb,   \hbox{where}  \ a_{\ell,t} \in \C  $. Therefore, we get $\sigma = A \jb $, where $A  = [ a_{\ell,t}]_{1\leq \ell,t \leq n} \in \mathrm{GL}(n,\C )$.
	
	{\it Case 2.~} Let $X$ has at least two distinct eigenvalue classes, i.e., $m<n$. Then by using  \thmref{lem-Jordan-sl(n,C)}, we can write Jordan decomposition of $X$ as follows:
	\begin{equation}
		X = [\x_{\ell,t}]_{1\leq \ell, t \leq n} = ( {a \ib }) \mathrm{I}_{m} \oplus D,   
	\end{equation} 
	where  $D = \mathrm {diag}(\lambda_{m+1}, \lambda_{m+2},\dots,\lambda_{n}), $ and $   \lambda_{\ell} \in \C $ has non-negative imaginary parts such that $  {a \ib } \neq \pm \lambda_\ell$ for all $ m+1 \leq \ell \leq n $. 
	From $ \sigma X \sigma^{-1} = - X $, we get 
	\begin{eqnarray*} 
		&	{ \sigma X} & =  - X {\sigma} \\
		&\Longleftrightarrow& \sum_{k=1}^n  ( \sigma_{\ell,k})( x_{k,t})  = \sum_{k=1}^n  (- x_{\ell,k} ) ( \sigma_{k,t}) \, \, \hbox{for all}  \ \ 1\leq \ell,t \leq  n.\\
		&\Longleftrightarrow & ( \sigma_{\ell,t})( x_{t,t}) = (- x_{\ell,\ell} )( \sigma_{\ell,t}) \,\, \hbox{for all }  1\leq \ell,t \leq n.
	\end{eqnarray*}
	This implies 
	\begin{align*}
		( \sigma_{\ell,t}) ({a \ib })&\, =\, (- {a \ib }) ( \sigma_{\ell,t}) ~ \qquad \text{ if  }\  1\leq \ell,\,t \leq m;\\
		( \sigma_{\ell,t}) (\lambda_t) &\, =\,  (- {a \ib }) ( \sigma_{\ell,t}) ~ \qquad \text{ if  } \ 1\leq \ell \leq m,  \ m +1\leq t \leq n ;\\
		( \sigma_{\ell,t}) ( {a \ib } ) &\, =\, (-\lambda_\ell) ( \sigma_{\ell,t}) ~ \qquad \text{ if  }\  m +1\leq \ell \leq n,  \ 1\leq t \leq m\,.	
	\end{align*}
	Since $  {a \ib } \neq \pm \lambda_\ell$ for all $ m+1 \leq \ell \leq n $,  the following hold : 
	$$
	\begin{cases}
		( \sigma_{\ell,t}) =  ( a_{\ell,t}) \jb,   \hbox{ where}  \ a_{\ell,t} \in \C   & \text{if  $ 1\leq \ell,t \leq m$}\,, \\
		\sigma_{\ell,t} =0 & \text{if  $ 1\leq \ell \leq m,  \ m +1\leq t \leq n$ }, \\
		\sigma_{\ell,t} =0  & \text{if   $ m +1\leq \ell \leq n,  \ 1\leq t \leq m $}.
	\end{cases}
	$$
	Hence  $\sigma$ has  the form 
	$$
	{\sigma} =
	\begin{pmatrix}
		A \jb &  \\
		&  B \\ 
	\end{pmatrix}\,,
	$$ 
	where $ B \in  \mathrm{GL}(n-m,\H) $ and $ A  = [ a_{\ell,t}]_{1\leq \ell,t \leq m} \in  \mathrm{GL}(m,\C)$.

	Since $ \sigma^2 = \mathrm{I}_{n}$,  in both cases, we have
	$ A \jb A \jb =   \mathrm{I}_m$. This implies  $ A \overline{A} = - \mathrm{I}_m$ as  the matrix $ A\in {\GL}(m,\C) $. 
	By considering the determinant in the above relation, we get $|\det(A)|^2\,=\, (-1)^m$. This   
	contradicts the assumption that $m$ is odd. This proves the lemma.
\end{proof}

Recall that if $X \in \mathfrak{g} $ is $ {\rm Ad}_{ G} $-real (resp.  strongly $ {\rm Ad}_{ G} $-real) and $X = X_s + X_n$ is the Jordan decomposition of $X$, then $X_s $ and $X_n$ are $ {\rm Ad}_{G} $-real (resp.  strongly $ {\rm Ad}_{ G} $-real). The converse of this statement may not be true. The following example demonstrates this.

\begin{example}\label{ex1}
	Let $X=\begin{pmatrix}
		\ib & 1\\
		0 & \ib\\
	\end{pmatrix} \in \mathfrak{sl}(2,\H)$. 
	Then  $X=X_s+X_n$ is the Jordan decomposition of $ X $,  where $X_s=\begin{pmatrix}
		\ib & 0\\
		0 & \ib\\
	\end{pmatrix}$ and  $X_n=\begin{pmatrix}
		0 & 1\\
		0 & 0\\
	\end{pmatrix}$.
	Take $\sigma=\begin{pmatrix}
		0 & \jb\\
		-\jb & 0\\
	\end{pmatrix}$ and $\tau=\begin{pmatrix}
		1 & 0\\
		0 & -1\\
	\end{pmatrix}$. Then $\sigma$ and $\tau$ are involutions in $\mathfrak{sl}(2,\H)$ such that $\sigma X_s\sigma^{-1}=- X_s$ and $\tau X_n \tau^{-1}=-X_n$.  Thus $X_s$ and $X_n$ are strongly ${\rm Ad}_{{\rm SL }(2,\H)} $-real.
	
	Now assume that $X$ is strongly ${\rm Ad}_{{\rm SL }(2,\H)} $-real. Let $g:=\begin{pmatrix}
		g_{1,1}& 	g_{1,2}\\
		g_{2,1}& 	g_{2,2}
	\end{pmatrix}$ be an involution 
	such that $gXg^{-1}=-X$. This implies 
	\begin{equation} \label{equ-example}
		gX_sg^{-1}=-X_s, gX_ng^{-1}=-X_n \ \hbox{and} \ g^2= \mathrm{I}_2.  \end{equation}
	Now $gX_ng^{-1}=-X_n \Rightarrow gX_n=-X_n g  \Rightarrow g_{2,1} =0 \ \hbox{and} \ g_{1,1} = -g_{2,2}$.  Therefore, $g^2= \mathrm{I}_2$ implies $ g_{1,1} ^2 = 1$, i.e., $g_{1,1} = \pm1$. But $gX_sg^{-1}=-X_s \Rightarrow g_{1,1}\ib =-\ib  g_{1,1}$. This is a contradiction. Hence, $X$ is not strongly ${\rm Ad}_{{\rm SL }(2,\H)} $-real. \qed
\end{example}

The following lemma is a generalization of Example \ref{ex1}.
\begin{lemma}
	Let $X \in \mathfrak{sl}(n,\H)$ be the Jordan block $  \mathrm{J}(n,a\ib) $, where $a \in \R, a>0$. Then $X$ is not strongly ${\rm Ad}_{{\rm SL }(n,\H)} $-real.
\end{lemma}
\begin{proof} Write $X = X_s + X_n$,  where $X_s= ({a\ib}) \mathrm{I}_n$ and $X_n=\mathrm{J}(n,0)$.
	Suppose that $X$ is strongly ${\rm Ad}_{{\rm SL }(n,\H)} $-real in $ \mathfrak{sl}(n,\H)$.  Therefore, there exists an involution $g:=[g_ {\ell,t}]_{1 \leq \ell, t \leq n} \in {\SL}(n,\H)$ such that $gXg^{-1}=-X$. This implies 
	\begin{equation} \label{equ-2-example}
		gX_sg^{-1}=-X_s, \quad gX_ng^{-1}=-X_n \  \   \hbox{  and }\ \ g^2= \mathrm{I}_n.  
	\end{equation}
	Now on equating the entries in the first column of the left-hand side and first column of the right-hand side of equation $gX_ng^{-1}=-X_n$, we obtain $g_{\ell,1}= 0 $ for all $ 2 \leq \ell \leq n$. Further,  $g^2= \mathrm{I}_n$ implies $ g_{1,1}^2=1$, i.e., $g_{1,1} = \pm1$.  But $gX_sg^{-1}=-X_s \Rightarrow g_{1,1}\ib =-\ib  g_{1,1}$. This is a contradiction. Hence, the lemma is proved.
\end{proof}

Although the Jordan block $\mathrm{J}(n,\ib) \in \mathfrak{sl}(n,\H)$ is not strongly ${\rm Ad}_{{\rm SL}(n,\H)}$-real, Jordan form in $\mathfrak{sl}(n,\H)$ having  such Jordan blocks can be ${\rm Ad}_{{\rm SL }(n,\H)} $-real.

\begin{lemma} \label{lem-jor-mat-strong}
	Let 
	$X:=
	\begin{pmatrix}
		\mathrm{J}(n,a\ib) &  \\
		&  \mathrm{J}(n,a\ib) \\ 
	\end{pmatrix} \in  \mathfrak{sl}(2n,\H) $ be the Jordan form,  where $a \in \R, a>0$.  Then $X$ is strongly ${\rm Ad}_{{\rm SL }(2n,\H)} $-real.
\end{lemma}

\begin{proof}
	Consider  $g :=\sigma \tau \in \mathrm{SL}(2n,\H) $, where $\sigma $ and $\tau$ are defined in the following way: $\sigma := \begin{pmatrix}
		& ( \mathrm{I}_{n})   \jb\\
		-( \mathrm{I}_{n} )  \jb &    \\ 
	\end{pmatrix}$ and $ \tau :=  \tau_1 \oplus  \tau_1 $ where $ \tau_1 =  \textnormal{diag} (1,  -1,  1,  -1, \cdots , (-1)^{n+1} \ )_{n \times n }. $  
	Equivalently, 
	$ g = \begin{pmatrix}
		& A\\
		-A & \\ 
	\end{pmatrix}$,  where $ A = \mathrm{diag} (j,-j,j,-j,\dots,(-1)^{n+1}j)_{n \times n }$.
	Note that $ \sigma \in \mathrm{SL}(2n,\H) $ is an involution such that $ \sigma X_s \sigma^{-1} = - X_s $ and $ \sigma X_n \sigma^{-1} = X_n$. Similarly,  $ \tau \in \mathrm{SL}(2n,\H) $ is an involution such that $ \tau X_s \tau^{-1} = X_s $ and $ \tau X_n \tau^{-1} = - X_n $.  Therefore, $gXg^{-1} = -X$.  Further,  since $\tau$ and $ \sigma $ are involutions such that $ \tau\sigma = \sigma \tau$,  it follows that  $g$ is an involution.  This proves the lemma.
\end{proof}

The following result will be needed to prove  Theorem  \ref{thm-strong-real-sl(n,H)}.
\begin{lemma} \label{lem-partition-sl(n,H)}
	Let $X  \in  \mathfrak{sl}(n,\H)$ be the Jordan form $ \mathrm{J}(\mathbf{d}(n),ai)$,  where $a \in \R, a>0$ and $ \mathbf{d}(n)= [ d_1^{t_{d_1}}, \dots,   d_s^{t_{d_s}} ] \in \mathcal{P}(n)$ be the corresponding partition.  Let $X$ be strongly ${\rm Ad}_{{\rm SL }(n,\H)} $-real. Then $t_{d_\ell}$ is even for all $ 1 \leq \ell \leq s$.
\end{lemma}

\begin{proof}
	Since $X$ is strongly ${\rm Ad}_{{\rm SL }(n,\H)} $-real, there exists an involution $ g \in \mathrm{SL}(n,\H) $ such that $gXg^{-1} = -X$.
	Write $X = X_s + X_n $, where $X_s = (a \ib) \mathrm{I}_n $ and $X_n =  \mathrm{N}(\mathbf{d}(n),0)$.  Then
	\begin{equation}\label{eq-1-lem-part-sl(n,H)}
		g ((a \ib) \mathrm{I}_n) =- ((a \ib) \mathrm{I}_n)  g, 
	\end{equation}
	\begin{equation}\label{eq-2-lem-part-sl(n,H)}
		g \,  ( \mathrm{N}(\mathbf{d}(n),0)) =- (\mathrm{N}(\mathbf{d}(n),0)) \,  g,  \, \ \hbox{and} \ g^2= \mathrm{I}_n. 
	\end{equation}
	
	Now from \eqref{eq-1-lem-part-sl(n,H)},  we get $ g= h \jb $, where $ h \in \mathrm{GL}(n,\C) $.  Since $ h \in \mathrm{GL}(n,\C) $, we have 
	\begin{equation} \label{eq-3-lem-part-sl(n,H)}
		h \jb = \jb \overline{h},  \quad  \hbox{ and  } \quad h\,  (\mathrm{N}(\mathbf{d}(n),0)) =- (\mathrm{N}(\mathbf{d}(n),0))  h. 
	\end{equation}
	
	Thus from   \eqref{eq-2-lem-part-sl(n,H)} and \eqref{eq-3-lem-part-sl(n,H)}, we have $ h \in \mathrm{GL}(n,\C) $ satisfying following properties:
	\begin{enumerate}
		\item $h \,  ( \mathrm{N}(\mathbf{d}(n),0)) =- (\mathrm{N}(\mathbf{d}(n),0)) \,  h$,
		\item $h \overline{h} = - \mathrm{I}_n$.
	\end{enumerate}
	The proof now follows from Lemma \ref{lem-nilpotent-reverser}.
\end{proof}

Next we will give a proof of \thmref{thm-strong-real-sl(n,H)}, which classifies the strongly  ${\rm Ad}_{{\rm SL }(n,\H)} $-real elements in $ \mathfrak{sl}(n,\H)$.

\noindent{\bf Proof of \thmref{thm-strong-real-sl(n,H)}.~}
First, we assume that $X$ satisfies condition (\ref{cond-2-main-thrm-sl(n,H)-str}) of \thmref{thm-strong-real-sl(n,H)}.  Suppose that distinct non-zero purely imaginary eigenvalues of $X$ are represented by $\mu_i$, $1 \leq i \leq \ell $ with multiplicities $m_{\mu_i}$. Let ${\d}({m_{\mu_i}})\in \mathcal{P}(m_{\mu_i}) $ be the corresponding partition to $ \mu_i $.  Note that each  ${m_{\mu_i}}$ is even. 

Since $X$ is real in $\mathfrak{sl}(n,\H) $, so it satisfies conditions (\ref{cond-1-real-sl(n,H)-thr}) and (\ref{cond-2-real-sl(n,H)-thr}) of Theorem \ref{thm-real-sl(n,H)}.  For $1 \leq j \leq r $,
let $\lambda_j$ denote the distinct eigenvalues of $X$ with non-zero real part and multiplicity $m_{\lambda_j}$.
For each $ j $, $-\lambda_j$ is also an eigenvalue of $X$ with the same multiplicity $m_{\lambda_j}$ and the same partition ${\d}({m_{\lambda_j}})$.  Let $p_o \in \N\cup \{0\} $ denote the multiplicity of the eigenvalue $ 0 $ of $X$ such that $p_o +q +2p =n$, where $ \sum_{i=1}^{\ell} m_{\mu_i}= q,  \ \sum_{j=1}^{r} m_{\lambda_j} = p$.  Note that $q$ is even.

Then Jordan decomposition of $X$ has the following form: 
$$X=
\begin{pmatrix}
	\mathrm{N}(\d(p_o),\,0) &  & \\
	&  A \\ 
	& & B
\end{pmatrix},$$
where $A =  \oplus_{ i=1}^{\ell}  \mathrm{J}({\d}({m_{\mu_i}}),\mu_i )$ and $B=  \Big( \oplus_{ j=1}^{r}  \mathrm{J}({\d}({m_{\lambda_j}}),\lambda_j ) \Big)  \bigoplus \Big( \oplus_{ j=1}^{r}  \mathrm{J}({\d}({m_{\lambda_j}}),-\lambda_j )\Big)$.

We see that  each ${\d}({m_{\mu_i}})$ satisfies condition (\ref{cond-2-main-thrm-sl(n,H)-str})  of \thmref{thm-strong-real-sl(n,H)}.
After rearranging the Jordan blocks in $A$ and using  Lemma \ref{lem-jor-mat-strong}, we can construct an involution  $g_1 \in \mathrm{SL}(q,\H)$ such that $g_1 A {g_1}^{-1}=-A$.  Similarly,  using  idea of the proof of Theorem \ref{thm-strongly-real-sl(n,C)},  we can construct an involution $g_2 \in \mathrm{SL}(2p,\H) $ such that $g_2 B {g_2}^{-1}=-B$.
Define $g := {g_{o}} \oplus g_1 \oplus g_2 $, where $g_o =  \textnormal{diag} (1,  -1,  1,  -1, \dots , (-1)^{{p_o} +1} \ )_{p_o \times p_o }. $  Then $g \in \mathrm{SL}(n,\H) $ is an involution such that $gXg^{-1} =- X$.   

Conversely,  let $X$ be strongly $ {\rm Ad}_{{\SL}(n,\H)} $-real such that $\mu$ is the unique complex representative of a purely imaginary non-zero eigenvalue class of $X$ with multiplicity $m_{\mu}$. Then $\mu = a \ib$, where $a \in \R, a>0$ .
Therefore, we can write Jordan decomposition of $X$ as:
$$
X= X_s + X_n, $$ 
such that $ X_s = (\mu \mathrm{I}_{m_\mu} ) \bigoplus \Big( \oplus_{ j=1}^{r} (\lambda_j \mathrm{I}_{m_{\lambda_j}} ) \Big) $
and $X_n = \mathrm{N}(\d(m_\mu ),\,\mu) \bigoplus  \Big( \oplus_{ j=1}^{r}     \mathrm{N}(\d(m_{\lambda_j},\,\lambda_j) \Big)$,
where $\lambda_j$ are distinct complex numbers with non-negative imaginary part such that $ [\mu] \neq [\lambda_j] $ for all  $1 \le  j \le r. $

Let $g=[g_{i,j}]_{1 \leq i, j \leq n} \in \mathfrak{sl}(n,\H)$ be an involution such that $gXg^{-1}=-X$. Then
\begin{equation} \label{eq-1-str-sl(n,H)}
	gX_sg^{-1}=-X_s,\,\  gX_ng^{-1}=-X_n \ \hbox{and} \  g^2= \mathrm{I}_n.
\end{equation}
Since $[\mu] \neq [\lambda_j]$ for all $ 1 \leq j \leq r$, from the relation $gX_sg^{-1}=-X_s$,  we get that $g$ has the following form: 
\begin{equation} \label{eq-2-str-sl(n,H)}
	g= \begin{pmatrix}
		g_1 &  \\
		&  g_2 \\ 
	\end{pmatrix}, \, \hbox{where } \,  g_1 \in {\GL}(m_{\mu}, \H), g_2 \in {\GL}((n- m_{\mu}), \H).
\end{equation}
From  \eqref{eq-1-str-sl(n,H)} and \eqref{eq-2-str-sl(n,H)}, we get that there exists an involution $g_1 \in {\SL}(m_{\mu}, \H)$ such that $(g_1) \Big( \mathrm{J}({\d}({m_{\mu}}),\mu ) \Big) (g_1^{-1}) = - \mathrm{J}({\d}({m_{\mu}}),\mu)$, where ${\d}({m_{\mu}})$ is partition corresponding  to $\mu$.  The proof now follows from Lemma \ref{lem-partition-sl(n,H)}.
\qed

\subsection*{Acknowledgment} 
The authors thank Peggy L. Currid for communications during the review process.

\end{document}